\documentclass{amsart}

\usepackage[utf8x]{inputenc}
\usepackage{amsmath}
\usepackage{amssymb}
\usepackage{amsthm}
\usepackage{graphicx}

\def\C{\mathcal{C}}
\def\N{\mathbb{N}}
\def\R{\mathbb{R}}
\def\f{{\tilde{f}}}
\def\g{{\tilde{g}}}

\theoremstyle{plain}
\newtheorem{theorem}{Theorem}
\newtheorem{proposition}[theorem]{Proposition}

\theoremstyle{definition}

\newtheorem{remark}[theorem]{Remark}

\title{Building Cantor's Bijection}

\author{Samuel Nicolay \and Laurent Simons}
\address{
Universit\'e de Li\`ege\\
 Institut de Math\'ematique (B37)\\
 Grande Traverse, 12\\
 B-4000 Li\`ege (Sart-Tilman), Belgium}
\email{S.Nicolay@ulg.ac.be, L.Simons@ulg.ac.be}


\begin{document}

\begin{abstract}
Cantor's first idea to build a one-to-one mapping from the unit interval to the unit square did not work since, as pointed out by Dedekind, the so-obtained function is not surjective. Here, we start from this function and modify it (on a negligible set) in order to obtain the desired result: a one-to-one correspondance between the unit interval and the unit square.
\end{abstract}

\maketitle

\section{Introduction.}
At the end of the 19th century, Cantor spent a lot of his time on proving the existence of one-to-one mappings between sets. In particular, as borne out by the epistolary relation with Dedekind \cite{dugac,hinkis}, he was concerned about finding such a correspondance between the set of natural numbers and the set of positive real numbers. Even if, following Dedekind, this work was only of theoretical interest, Cantor showed in 1874 that there does not exist any bijection between the set of all natural numbers and the unit interval \cite{cantor:1874}. Such a result paved the way for the set theory.

Once this problem solved, Cantor addressed to Dedekind a question that can be resumed as follows: ``Can a surface (e.g.\ the unit square) be put into relation with a curve (e.g.\ the unit segment)?'' \cite{dugac,hinkis}. The lack of rigorous definition for the notion of dimension did not help to fully understand the problem. In 1877, Cantor proposed the following example, based on the (unique) proper decimal expansion of the real numbers. If $x$ and $y$ both belong to the unit segment $[0,1)$, let us suppose that one has
\[
 x = \sum_{k=1}^{+\infty} \frac{x_k}{10^k} = 0.x_1 x_2 \cdots
 \qquad\text{and}\qquad
 y = \sum_{k=1}^{+\infty} \frac{y_k}{10^k} = 0.y_1y_2\cdots,
\]
where such expansions are supposed to be proper (e.g.\ there does not exist $k_0$ for which $k>k_0$ implies $x_k=9$) and let $\C$ be the function defined as
\[
 \C: [0,1)^2\to [0,1) \quad (x,y)\mapsto \sum_{k=1}^{+\infty} \frac{x_k}{10^{2k-1}} + \sum_{k=1}^{+\infty} \frac{y_k}{10^{2k}} = 0.x_1y_1x_2y_2x_3y_3\cdots.
\]
Dedeking objected that such a function is not surjective, since a number of the form
\[
 z
 = \sum_{k=1}^l \frac{z_k}{10^k}+ \sum_{k=1}^{+\infty} \frac{9}{10^{l+2k-1}} +\sum_{k=1}^{+\infty} \frac{z_{l+2k}}{10^{l+2k}} 
 =0.z_1 z_2\cdots z_l 9 z_{l+2} 9 z_{l+4} 9 \cdots
\]
has no preimage under $\C$: if $l$ is even, there is no $x$ such that $\C (x,y)=z$ and if $l$ is odd, there is no $y$ such that $\C (x,y)=z$. Cantor overcame this problem by replacing the decimal expansion in $\C$ with the expansion in terms of continued fractions. His work was published in \cite{cantor:1877}, with a praragraph explaining why his first idea could not work (and omitting any reference to Dedekind). This paper helped to clarify the confusion between dimension and cardinality. Cantor's bijection defined using continued fractions has been studied in~\cite{NiSi}.

In this paper, we go back on Cantor's first idea to build a one-to-one mapping between the unit segment and the unit square. We start from the function $\C$ relying on the decimal expansion and use the Schr\"oder-Bernstein theorem to define the desired bijection. This theorem was first conjectured by Cantor and independently proved by Bernstein and Schr\"oder in 1896 (see \cite{bernstein,cantor:1932,schroeder}, let us also notice that other names, such as Dedekind, should be added to this list). In other words, Cantor's first idea could have led to the craved mapping, but he did not have such a result at the time he was working on the topic; it would be conjectured by himself a few years later in \cite{cantor:1887}. Before building the bijective function in Section~3, we recall the Schr\"oder-Bernstein theorem and give a classical proof that will be used in the sequel.

\section{A ``practical'' proof of Schr\"oder-Bernstein theorem.}

There exist several proofs of Schr\"oder-Bernstein theorem \cite{hinkis} (the most classical ones use Tarski's fixed point theorem, or follow the idea of Richard Dedekind \cite{dedekind} or Julius K{\"o}nig \cite{konig}). The advantage of the one we present below (which is well-known and inspired by ideas of~\cite{bernstein:1906,reichbach}) is that it explicitly shows how to build a bijection between two non-empty sets, starting from injections between these sets.

\begin{theorem}[Schr\"{o}der-Bernstein]\label{SB}
Let $A$ and $B$ be non-empty sets. If there exist an injection from $A$ to $B$ and an injection from $B$ to $A$, then there exists a bijection from $A$ to $B$.
\end{theorem}

\begin{proof}
Let $f$ be an injection from $A$ to $B$ and $g$ be ab injection from $B$ to $A$. We define the sequences $(A_n)_{n\in\N}$ of subsets of $A$ and $(B_n)_{n\in\N}$ of subsets of $B$ as follows:
\begin{equation}\label{suites}
\left\{
\begin{tabular}{ll}
$A_0 = A\setminus g(B)$\\
$B_n = f(A_n)$, & for $n\in\N$\\
$A_n = g(B_{n-1})$, & for $n\in\N\setminus\{0\}$
\end{tabular}
\right..
\end{equation}
If $A_0=\emptyset$, then $g(B)=A$ and thus $g$ is surjective. The application $g^{-1}$ is then a bijection from $A$ to $B$. Therefore, we can assume that $A_0$ is not the empty set.

Since $f$ and $g$ are injective, none of the elements of the sequences $(A_n)_{n\in\N}$ and $(B_n)_{n\in\N}$ are empty and thus
\[
\bigcup_{n\in\N}A_n\neq\emptyset,\quad\bigcup_{n\in\N}B_n\neq\emptyset\quad\text{and}\quad
f\left(\bigcup_{n\in\N}A_n\right)\neq\emptyset.
\]
Moreover, we have
\[
f\left(\bigcup_{n\in\N}A_n\right)\subseteq\bigcup_{n\in\N}B_n
\]
and the restriction $\f$ of $f$ to $\bigcup_{n\in\N}A_n$ is clearly a bijection from $\bigcup_{n\in\N}A_n$ to $\bigcup_{n\in\N}B_n$.

If $B=\bigcup_{n\in\N} B_n$, then $A=\bigcup_{n\in\N}A_n$ because $f$ is injective and thus $\f$ is a bijection from $A$ to $B$.

Let us now assume that $B\setminus\bigcup_{n\in\N}B_n\neq\emptyset$. Since $g$ is injective, we have
\[
 g\left(B\setminus\bigcup_{n\in\N}B_n\right)\subseteq A\setminus\bigcup_{n\in\N}A_n
\]
and $A\setminus\bigcup_{n\in\N}A_n\neq\emptyset$. Let us denote $\g$ the restriction of $g$ to $B\setminus\bigcup_{n\in\N}B_n$ and show that $\g$ is a bijection from $B\setminus\bigcup_{n\in\N}B_n$ to $A\setminus\bigcup_{n\in\N}A_n$. It is clear that $\g$ is injective. Since
\[
A\setminus\bigcup_{n\in\N}A_n=g(B)\cap\left(\bigcap_{n\in\N_0}(A\setminus g(B_{n-1}))\right)
=g(B)\cap\left(A\setminus g\left(\bigcup_{n\in\N}B_n\right)\right),
\]
$\g$ is also surjective. 

It only remains to put the pieces together in order to build a bijection from $A$ to $B$. Since $\f$ is a bijection from $\bigcup_{n\in\N}A_n$ to $\bigcup_{n\in\N}B_n$ and $\g^{-1}$ is a bijection from $A\setminus\bigcup_{n\in\N}A_n$ to $B\setminus\bigcup_{n\in\N}B_n$, the application $h$ defined by
\[
h(a)=\left\{
\begin{tabular}{ll}
$\f (a)$ & if $a\in \displaystyle\bigcup_{n\in\N}A_n$ \\
$\g^{-1}(a)$ & if $a\in A\displaystyle\setminus\bigcup_{n\in\N}A_n$
\end{tabular}
\right.
\]
is a bijection from $A$ to $B$, hence the conclusion.
\end{proof}

\begin{remark}
Let us note that the definition of the function $h$ given above is non-constructive \cite{troelstra}: there is no general method to decide whether or not an element of $A$ belongs to $\bigcup_{n\in \N} A_n$ in a finite number of steps. However, in the specific case we will consider, the problem becomes simpler.
\end{remark}

\section{A bijection between the unit square and the unit segment based on the decimal expansion.}

Let us build a bijection between the unit square $[0,1]^2$ and the unit segment $[0,1]$ starting from the function $\mathcal{C}$. Since the construction is entirely based on the proof of the previous theorem, we will use the notations of this proof.
\medskip

Let us set $A=[0,1]^2$, $B=[0,1]$ and let $f$ be the function defined by
\[
f(x,y)=\left\{\begin{tabular}{ll}
$\displaystyle  \sum_{k=1}^{+\infty}\frac{x_k}{10^{2k-1}}+\sum_{k=1}^{+\infty}\frac{y_k}{10^{2k}}=0.x_1y_1x_2y_2x_3y_3\cdots$ & if $(x,y)\in[0,1)^2$ \\
$\displaystyle  \sum_{k=1}^{+\infty}\frac{9}{10^{2k-1}}+\sum_{k=1}^{+\infty}\frac{y_k}{10^{2k}}=0.9y_19y_29y_3\cdots$ & if $(x,y)\in\{1\}\times [0,1)$ \\
$\displaystyle  \sum_{k=1}^{+\infty}\frac{x_k}{10^{2k-1}}+\sum_{k=1}^{+\infty}\frac{9}{10^{2k}}=0.x_19x_29x_3\cdots$ & if $(x,y)\in[0,1)\times\{1\}$ \\
$1$ & if $(x,y)=(1,1)$
\end{tabular}\right.
\]
where $(x_k)_{k\in\N\setminus\{0\}}$ and $(y_k)_{k\in\N\setminus\{0\}}$ are the proper decimal expansions of the real numbers $x$ and $y$ of $[0,1)$. In fact, we have $f(x,y)=\mathcal{C}(x,y)$ for $(x,y)\in [0,1)^2$, so that $f$ is simply an extension of $\C$ to $[0,1]^2$. Let $g$ be the function defined by $g(t)=(t,0)$ for $t\in B$. It easy to check that both $f$ and $g$ are injective.

Let us construct the sequences $(A_n)_{n\in\N}$ and $(B_n)_{n\in\N}$ step by step as in \eqref{suites}. For $n=0$, we have
\[
A_0=A\setminus g(B)=[0,1]\times (0,1]
\]
and
\[
B_0=f(A_0)=\{1\}\cup\{t\in[0,1):t_{2k}\neq 0\;\text{for some $k\in\N\setminus\{0\}$}\},
\]
where $(t_k)_{k\in\N\setminus\{0\}}$ is the proper decimal expansion of the real number $t$ belonging to $[0,1)$.

For $n=1$, we directly have $A_1=g(B_0)=B_0\times\{0\}$. In order to construct $B_1=f(A_1)$, let us take $(x,0)\in A_1$. We have $x_{2k}\neq0$ for some $k\in\N\setminus\{0\}$ by definition of $B_0$ and thus
\[
f(x,0)=\left\{\begin{tabular}{ll}
$\displaystyle\sum_{k=1}^{+\infty}\frac{9}{10^{2k-1}}=0.909090\cdots$ & if $x=1$ \\
$\displaystyle\sum_{k=1}^{+\infty}\frac{x_k}{10^{2k-1}}=0.x_10x_20x_30\cdots$ & if $x\neq 1$
\end{tabular}\right..
\]
We can then write
\[
f(x,0)=\sum_{k=1}^{+\infty}\frac{s_k}{10^{2k-1}}=0.s_10s_20s_30\cdots
\]
where $(s_k)_{k\in\N\setminus\{0\}}$ is a sequence satisfying only one of the two following conditions: 
\begin{itemize}
\item $s_k=9$ for all $k\in\N\setminus\{0\}$,
\item $(s_k)_{k\in\N\setminus\{0\}}$ is the proper decimal expansion of a real number of $[0,1)$ and $s_{2k}\neq0$ for some $k\in\N\setminus\{0\}$.
\end{itemize}
We will denote by $S$ the set of sequences which satisfy one of the two previous conditions. 
We therefore have
\[
B_1=\left\{t\in[0,1):t=\sum_{k=1}^{+\infty}\frac{s_k}{10^{2k-1}}\;\text{with $(s_k)_{k\in\N\setminus\{0\}}\in S$}\right\}.
\]

For $n=2$, the argument is similar. We have $A_2=g(B_1)=B_1\times\{0\}$. If $(x,0)\in A_2$, then $x_{2k}=0$ for all $k\in\N\setminus\{0\}$ and $x_{4k-1}\neq 0$ for some $k\in\N\setminus\{0\}$. Consequently, we have
\[
f(x,0)=\sum_{k=1}^{+\infty}\frac{x_{2k-1}}{10^{4k-3}}=0.x_1000x_3000x_5000\cdots
\]
and so
\[
B_2=\left\{t\in[0,1):t=\sum_{k=1}^{+\infty}\frac{s_k}{10^{4k-3}}\;\text{with $(s_k)_{k\in\N\setminus\{0\}}\in S$}\right\}.
\]

By going on in this way, we obtain $A_n=B_{n-1}\times\{0\}$ and
\[
B_n=\left\{t\in[0,1):t=\sum_{k=1}^{+\infty}\frac{s_k}{10^{2^nk-(2^n-1)}}\;\text{with $(s_k)_{k\in\N\setminus\{0\}}\in S$}\right\},
\]
for all $n\in\N\setminus\{0\}$.

Since $A_0\neq\emptyset$, $B\setminus\bigcup_{n\in\N}B_n\neq\emptyset$ (we have $0\notin B_n$ for any $n\in\N$) and $g^{-1}(x,y)=x$ for $(x,y)\in A\setminus\bigcup_{n\in\N}A_n$, we have proved the following proposition thanks to Theorem~\ref{SB}.

\begin{proposition}
The function $f^*$ defined by
\[
f^*(x,y)=\left\{\begin{tabular}{ll}
$f(x,y)$ & if $(x,y)\in\displaystyle\bigcup_{n\in\N} A_n$\\
$x$      & otherwise
\end{tabular}\right.
\]
is a bijection from $[0,1]^2$ to $[0,1]$.
\end{proposition}

\begin{remark}
As expected, we have $f^*=f$ almost everywhere on $[0,1]^2$ (with respect to the Lebesgue measure), since the set $[0,1]^2\setminus\bigcup_{n\in\N}A_n$ is included in the segment $[0,1]\times\{0\}$, which is a negligible set in~$\R^2$. Therefore, we have $\C=f^*$ almost everywhere.
\end{remark}

\bibliography{cantor}{}
\bibliographystyle{plain}

\end{document}